\documentclass{amsart}
\usepackage[utf8]{inputenc}
\usepackage{amssymb,latexsym}
\usepackage{amsmath}
\usepackage{graphicx}
\usepackage{textcomp}

\usepackage{amsthm,amssymb,enumerate,graphicx, tikz}
\usepackage{amscd}
\usepackage{setspace}
\usepackage{comment}
\usepackage{hyperref}
\usepackage{cleveref}
\usepackage{subcaption}

\newtheorem{theorem}{Theorem}[section]

\newtheorem*{theorem*}{Theorem}

\newtheorem{lemma}[theorem]{Lemma}

\newtheorem{claim}[theorem]{Claim}

\theoremstyle{definition}
\newtheorem{definition}[theorem]{Definition}

\theoremstyle{remark}

\newtheorem{example}[theorem]{Example}

\newcommand{\cf}{\mathcal{F}}

\newcommand{\F}{\mathcal{F}}

\newcommand{\M}{\mathcal{M}}

\newcommand{\conv}{\textrm{conv}}
\newcommand{\supp}{\textrm{supp}}

\title{A Sparse colorful polytopal KKM Theorem}
\author{Daniel McGinnis}\thanks{D. McGinnis: Department of Mathematics, Iowa State University, USA.  \url{dam1@iastate.edu}.} 
\author{Shira Zerbib}\thanks{S. Zerbib: Department of Mathematics, Iowa State University, USA.  \url{zerbib@iastate.edu}. S. Zerbib was supported by NSF grant DMS-1953929.}

\begin{document}

\begin{abstract}
    Recently Sober\'on \cite{soberon2021fair} proved a far-reaching generalization of the colorful KKM Theorem due to Gale \cite{galeequilibrium1984}: let $n\geq k$, and assume that a family of closed sets $(A^i_j\mid i\in [n], j\in [k])$ has the property that for every $I\in \binom{[n]}{n-k+1}$, the family $\big(\bigcup_{i\in I}A^i_1,\dots,\bigcup_{i\in I}A^i_k\big)$ is a KKM cover of the $(k-1)$-dimensional simplex $\Delta^{k-1}$; then there is an injection $\pi:[k] \rightarrow [n]$ so that $\bigcap_{i=1}^k A_i^{\pi(i)}\neq \emptyset$. We prove a polytopal generalization of this result, answering a question of Sober\'on in the same note. We also discuss applications of our theorem to fair division of multiple cakes, $d$-interval piercing, and a generalization of the colorful Carath\'eodory theorem. 
\end{abstract}

\maketitle

\section{Introduction}

The KKM theorem due to Knaster, Kurotowski, and Mazurkiewicz \cite{knaster1929} is a set covering version of Sperner's lemma \cite{spernerslemma1928} and  Brouwer's fixed point theorem. 
It states that if there exists a family of closed subsets $(A_1,\dots,A_k)$ of the $(k-1)$-dimensional simplex 
$\Delta^{k-1} =\conv\{v_1,\dots,v_k\}$ that satisfy $\sigma \subset \bigcup_{v_i \in \sigma} A_i$ for 
every face $\sigma$ of $\Delta^{k-1}$, then $\bigcap_{i=1}^k A_i \neq \emptyset$. 
A family of sets $(A_1,\dots,A_k)$ satisfying the conditions of the KKM theorem is called a {\em KKM cover} of  $\Delta^{k-1}$.

There are many generalizations of the KKM theorem. One such result is the colorful KKM theorem by Gale \cite{galeequilibrium1984}, stating that given $k$ KKM covers $(A^i_1,\dots,A^i_k)$ for  $1\leq i\leq k$, there exists a permutation $\pi: [k] \rightarrow [k]$ so that $\bigcap_{i=1}^k A_i^{\pi(i)}\neq \emptyset$.

Another generalization is the KKMS theorem, due to Shapley \cite{shapleybalanced1973}. A {\em KKMS cover} of $\Delta^{k-1}$  is a collection of closed subsets $(A_\tau \mid \tau \text{ a face of } \Delta^{k-1})$ such that for every face $\sigma$ of $\Delta^{k-1}$ we have  $\sigma\subset \bigcup_{\tau\subset \sigma}A_\tau$. For a face $\sigma$ of $\Delta^{k-1}$, let $b(\sigma)$ denote the barycenter of $\sigma$. The KKMS theorem states that given a KKMS cover of $\Delta^{k-1}$, there exist faces $\tau_1,\dots,\tau_k$ of $\Delta^{k-1}$ so that $b(\Delta^{k-1})\in\conv\{b(\tau_1),\dots,b(\tau_k)\}$ and  $\bigcap_{i=1}^k A_{\tau_i}\neq \emptyset$. Shih and Lee  \cite{shihcombinatorial1993} proved a colorful version of the KKMS theorem.

Komiya proved a generalization of the KKMS Theorem for general polytopes. 

\begin{theorem}[Komiya \cite{Komiyasimple1994}]
\label{thm:komiya}
Let $P$ be a $(k-1)$-dimensional polytope with $p\in P$. Suppose that for every non-empty proper face $\tau$ of $P$ we are given a closed subset $A_\tau$ of $P$ and a point $y_\tau\in \tau$. If $\sigma\subset \bigcup_{\tau\subset \sigma}A_\tau$ for every face $\sigma$ of $P$,  then there exist faces $\tau_1,\dots,\tau_k$ of $P$ such that $p\in \conv\{y_{\tau_1}\dots,y_{\tau_k}\}$ and $\bigcap_{i=1}^k A_{\tau_i}\neq \emptyset$.
\end{theorem}

A family of sets $(A_\tau \mid \tau \text{ a face of } P)$ satisfying the condition of Theorem \ref{thm:komiya} is called a {\em Komiya cover} of $P$. Note that the theorem is true also if all the sets $A_\tau$ are open (see e.g. \cite{floriancolorful2019}). 

Frick and Zerbib \cite{floriancolorful2019} proved a colorful extension of Theorem \ref{thm:komiya}.

\begin{theorem}[Frick-Zerbib \cite{floriancolorful2019}]\label{frick-zerbib}
Let $P$ be a $(k-1)$-dimensional polytope with $p\in P$. Suppose for every non-empty proper face $\tau$ of $P$ we are given $k$ closed (open) subsets $A^1_\tau,\dots, A^k_\tau$ of $P$ and $k$ points $y^1_\tau,\dots,y^k_\tau\in \tau$. If for every $i\in [k]$ and every face $\sigma$ of $P$, we have $\sigma\subset \bigcup_{\tau\subset \sigma}A^i_\tau$, then there exist faces $\tau_1,\dots,\tau_k$ of $P$ such that $0\in \conv\{y^1_{\tau_1}\dots,y^k_{\tau_k}\}$ and $\bigcap_{i=1}^k A^i_{\tau_i}\neq \emptyset$.
\end{theorem}

Recently, Sober\'on \cite{soberon2021fair} proved a beautiful generalization of the colorful KKM Theorem that we call here the {\em sparse colorful KKM theorem}. Let $n\geq k,m$. A family of closed (open) sets $(A^i_1,\dots, A^i_k \mid i\in [n])$ forms a {\em $m$-weakly KKM cover} of $\Delta^{k-1}$, if for every $I\in \binom{[n]}{n-m+1}$, the sets $\Big(\bigcup_{j\in I}A^j_1,\dots,\bigcup_{j\in I}A^j_k\Big)$ form a KKM cover of $\Delta^{k-1}$. The notion of $m$-weakly KKM cover was defined by Sober\'on in \cite{soberon2021fair}, where he proved:

\begin{theorem}[Sober\'on \cite{soberon2021fair}]\label{soberon}
Let $n\geq k$ be positive integers. Assume the family $(A^i_1,\dots, A^i_k \mid i\in [n])$ forms a $k$-weakly KKM cover of $\Delta^{k-1}$. Then there is an injection $\pi:[k] \rightarrow [n]$ so that $\bigcap_{i=1}^k A_i^{\pi(i)}\neq \emptyset$.
\end{theorem}

In the same note, Sober\'on  asked whether Theorem \ref{soberon} can be extended to general polytopes, in the same way  that Theorem \ref{frick-zerbib} generalizes the colorful KKMS theorem. In this paper we positively answer this question.  

We first give a  definition of {\em $m$-weakly Komiya cover} for a general polytope $P$. Given a polytope $P$, let $F(P)$ denote the set of non-empty, proper faces of $P$.

\begin{definition}[$m$-weakly Komiya cover] Let $n\ge k,m$ be positive integers, and 
let $P$ be a $(k-1)$-dimensional polytope. A family of closed sets $(A^i_\sigma \mid i\in [n], \sigma \in F(P))$, is called an \textit{$m$-weakly Komiya cover} of $P$ if for every $I\in \binom{[n]}{n-m+1}$, the family $\Big(\bigcup_{j\in I}A^j_\sigma \mid \sigma \in F(P)\Big)$ is a Komiya cover of $P$.
\end{definition}

 We are now ready to state our main theorem.
 
 \begin{theorem}\label{thm:main}
 Let $n\ge k\ge 2$ be integers. Let $P$ be a $(k-1)$-dimensional polytope with $p\in P$. Assume that for every  $\tau\in F(P)$, we are given $n$ points $y_\tau^1,\dots,y_\tau^n\in \tau$. If the family  $(A^i_\tau \mid i\in [n], \tau \in F(P))$  forms a $k$-weakly Komiya cover of $P$, then there exists an injection $\pi: [k]\rightarrow [n]$ and faces $\tau_{1},\dots,\tau_{k}$ such that $p\in  \conv\{y_{\tau_{1}}^{\pi(1)},\dots,y_{\tau_{k}}^{\pi(k)}\}$ and $\bigcap_{i=1}^k A^{\pi(i)}_{\tau_{i}}\neq \emptyset$.
 \end{theorem}
 
 As in the case of the previous theorems, Theorem \ref{thm:main} is true also if all the sets $A^i_\tau$ are open.
 
Let us discuss applications of Theorem \ref{thm:main}. 
The first is a {\em sparse colorful piercing theorem} for {\em hypergraphs of $d$-intervals}. A \textit{$d$-interval} is union of $d$ closed intervals in $\mathbb{R}$. A \textit{separated $d$-interval} is $d$-interval consisting $d$ disjoint interval components $h^1,\cdots, h^d$ such that $h^{i+1}\subset (i,i+1)$ for $0\leq i\leq d-1$. A \textit{hypergraph of (separated) $d$-intervals} is a hypergraph $H$ whose vertex set is $\mathbb{R}$ and whose edge set is a finite family of (separated) $d$-intervals.

A \textit{matching} in a hypergraph $H$ is a set of disjoint edges. The \textit{matching number} of $H$, which we denote as $\nu(H)$, is the maximum size of a matching. A \textit{cover} of $H$ is a set of vertices that intersects each edge of $H$, and the \textit{covering number} (or {\em piercing number}) of $H$, $\tau(H)$ is the minimum size of a cover.

Tard\'os \cite{tardos1995} and Kaiser \cite{Kaiser} proved the following result on the ratio between the covering and matching numbers in hypergraphs of $d$-intervals.

\begin{theorem}[Tard\'os\cite{tardos1995}, Kaiser \cite{Kaiser}]\label{tardos-kaiser}
In any hypergraph of $d$-intervals  $H$ we have $\tau(H)\leq (d^2-d+1)\nu(H)$. If $H$ is a hypergraph of separated $d$-intervals, then $\tau(H)\leq (d^2-d)\nu(H)$.
\end{theorem}

Matou\v{s}ek \cite{matouseklower2001} constructed examples of hypergraphs of $d$-intervals $H$ such that $\tau(H)=\Omega(\frac{d^2}{\log(d)}\nu(H))$, showing that the bounds from Theorem \ref{tardos-kaiser} are not far from optimal. Aharoni, Kaiser, and Zerbib \cite{AharoniFractional2017} gave an alternative proof of Theorem \ref{tardos-kaiser} using The KKMS theorem and Theorem \ref{thm:komiya}. Frick and Zerbib \cite{floriancolorful2019} applied Theorem \ref{frick-zerbib} in a similar way to prove a colorful generalization of Theorem \ref{tardos-kaiser}. In the following, a colorful matching $\M$ in a family of edges $(\F_1,\dots,\F_k)$ in a hypergraph $H$ is a matching in $H$  in which $|\M\cap \F_i|\leq 1$ for every $i\in [k]$.

\begin{theorem}[Frick-Zerbib \cite{floriancolorful2019}]\label{thm:colorful-intervals} \hfill
\begin{enumerate}
    \item  For  $i\in [k]$, let $\F_i$ be a hypergraph of $d$-intervals. If $\tau(\cf_i)\ge k$ for every $i\in [k]$, then there exists a colorful matching $\M$ of $(\F_1,\dots,\F_k)$  of size $|\M|\geq \frac{k}{d^2-d+1}$.
    \item For  $i\in [(k-1)d+1]$, let $\F_i$ be a hypergraph of separated $d$-intervals. If  $\tau(\F_i)\geq (k-1)d+1$ for every $i\in [k]$, then there exists a colorful matching $\M$ of $(\F_1,\dots,\F_{(k-1)d+1})$ of size $|\M|\geq \frac{k}{d-1}$. 
\end{enumerate}
\end{theorem}

Here we use Theorem \ref{thm:main} to obtain a ``sparse colorful" generalization of Theorem \ref{thm:colorful-intervals}. 

\begin{theorem}\label{thm:sparse colorful-intervals}
Let $n,k,m,d$ be positive integers..
\begin{enumerate}
    \item  Let $n\geq k$, and for every $i\in [n]$, let $\F_i$ be a hypergraph of $d$-intervals. If $\tau(\bigcup_{i\in I}\cf_i)\ge k$ for every  $I\in \binom{[n]}{n-k+1}$, then there exists a colorful matching $\M$ of $(\F_1,\dots,\F_n)$ of size $|\M|\geq \frac{k}{d^2-d+1}$.
    \item Let $n\geq (m-1)d+1$, and for every $i\in [n]$, let $\F_i$ be a hypergraph of separated $d$-intervals. Suppose  $\tau(\bigcup_{i\in I} \F_i)\geq (m-1)d+1$ for every $I\in \binom{[n]}{n-(m-1)d}$. Then there exists a colorful matching $\M$ of $(\F_1,\dots,\F_n)$ of size $|\M|\geq \frac{m}{d-1}$. 
\end{enumerate}
\end{theorem}

Another application of Theorem \ref{thm:main} is to fair division of multiple cakes.  Suppose that there are $n$ participants (players) at a party where $d$ cakes (which we identify with $d$ copies of the $[0,1]$ interval) are served.
Given any partition of cakes into $m$ interval pieces each,  every player can choose their favorite  $d$-tuple of pieces (a $d$-tuple of pieces contains one piece from each cake). The goal is to find such a partition so that every member in a large subset of players receives their favorite $d$-tuple.

We say that the set of players are $(m,d)$-hungry if the following two conditions hold:
\begin{enumerate}
    \item Given any partition of the $d$ cakes into $m$ interval pieces each, and given any set $I$ of $n-d(m-1)$ players, there is a player in $I$ that prefers a $d$-tuple of non-empty pieces. 
    \item The preference sets of the players are closed: if a player prefers some $d$-tuple of pieces in a converging sequence of partitions, they prefer the same $d$-tuple also in the limit partition. 
\end{enumerate}

Sober\'on \cite{soberon2021fair} used Theorem \ref{soberon} to prove the following generalization of the classical fair division theorem due to Stromquist \cite{stromquisthow1980} and Woodall \cite{woodalldividing1980}. 

\begin{theorem}[Sober\'on, \cite{soberon2021fair}]
Let $n\ge k$, and assume that $n$ players are $(k,1)$ hungry. Then there exists a partition of a single cake into $k$ interval pieces where $k$ players get their favorite piece.   
\end{theorem}

Here we give an extension of this theorem to fair division of multiple cakes. Our theorem is also a generalization of Theorem 2.1 in \cite{NSZ}.

\begin{theorem}\label{thm:cakes}
Let $n,m,d$ be positive integers so that  $n\ge  d(m-1)+1$.  Assume that  $n$ players are $(m,d)$-hungry. Then, there
exists a partition of the $d$ cakes into $m$ interval pieces each, and an allocation of $d$-tuples of  pieces  to a subset $P$ of players of size at least $ \frac{m}{d-1} $, so that each player in $P$ receives one of their favorite $d$-tuple of pieces.
\end{theorem}

A third application of Theorem \ref{thm:main} concerns a generalization of the colorful Carath\'eodory theorem \cite{baranygeneralization1982}. 
The following sparse colorful version of the colorful Carath\'eodory theorem was proven in \cite{HolmsenIntersection2016} and  \cite{soberonRobust2018}.

\begin{theorem}[Holmsen \cite{HolmsenIntersection2016}, Sober\'on \cite{soberonRobust2018}]\label{thm:sparseCaratheodory}
Let $k$ be a positive integer and $n \geq {k-1}$. If $X_1,\dots, X_n$ are finite subsets of $\mathbb{R}^{k-1}$ and for each $I\in \binom{[n]}{n-k+1}$, $0\in \conv(\bigcup_{i\in I}X_i)$, then there exist indices $i_1,\dots,i_k$ and points $x_j\in X_{i_j}$ such that $0\in \conv( \{x_1,\dots,x_k \})$.
\end{theorem}

In  \cite{floriancolorful2019}, Theorem \ref{frick-zerbib} was used to give a new proof of the colorful Carath\'eodory theorem. Using a similar approach,  Theorem \ref{thm:main} can be used to give another proof of Theorem \ref{thm:sparseCaratheodory}. The proof is essentially that same as the proof in \cite{floriancolorful2019}, so we omit it.

The paper is organized as follows.  Section 2 contains some preliminaries. In Section 3 we prove a theorem concerning the existence of certain triangulations that are needed for the proof of Theorem  \ref{thm:main}, and then in Section 4 we prove Theorem  \ref{thm:main}. Finally, in Section 5 we prove Theorems \ref{thm:sparse colorful-intervals} and \ref{thm:cakes}.

\section{Preliminaries}

All the triangulations considered in this paper are convex, that is, every face is the convex hull of its vertices. 
In order to prove Theorem \ref{thm:main}, we first need a vertex labelling version of Theorem \ref{thm:komiya}. Let $P$ be a polytope and $T$ a triangulation of $P$. For $v\in P$ denote by $\supp(v)$ the {\em support of $v$}, that is, the minimal face of $P$ containing $v$. A \textit{Sperner-Shapley labelling} of $T$ is an assignment $\lambda: V(T) \rightarrow F(P)$ such that $\lambda(v)\subset \supp(v)$ for every  $v\in V(T)$. 

The following  was essentially proved in \cite{floriancolorful2019}, and we give the proof here for completion. 
\begin{theorem}\label{sperner-shapley}
Let $P\subset \mathbb{R}^d$  be a polytope with $p\in P$, and let $T$ be a triangulation of $P$. Let $\lambda: V(T) \rightarrow F(P)$ be a Sperner-Shapley labeling of $T$. Suppose that for every $v\in V(T)$, a point $y(v)\in \lambda(v)$ is assigned.
Then there is a face $\tau$ of $T$ such that $p\in  \conv\{y(v) \mid  v\in V(\tau)\}$.
\end{theorem}
\begin{proof}
	By the Sperner-Shapley labeling condition we have $y(v)\in \lambda(v) \subset \supp(v)$.
	Extending $y$ linearly onto faces of $T$ defines a continuous map $Y\colon P \to P$, so that $Y(\sigma) \subset \sigma$ for every face $\sigma$ of $P$. This implies that $Y$ is 
	homotopic to the identity on~$\partial P$, and thus $Y|_{\partial P}$ has degree one. Then $Y$ is surjective 
	and we can find a point $x \in P$ such that $Y(x) = p$. Let $\tau$ be a face of $T$ containing~$x$.
	By the definition of $Y$,  we have  $p\in  Y(\tau)=\conv\{y(v) \mid v \in V(\tau)\}$.
\end{proof}

We will use the following simple lemma frequently. 

\begin{lemma}\label{maychoose}
Suppose the conditions of Theorem \ref{thm:main} hold. Let $v\in P$ and $I \subset [n]$. If $|I|\ge n-k+1$ then there exists $i\in I$ and $\tau \subset \supp(v)$ such that $v\in A^i_\tau$.
\end{lemma}
\begin{proof}
Let $J$ be a subset of $I$ of size $n-k+1$. Since the sets $A^i_\sigma$ form a $k$-weakly Komiya cover of $P$, the family $(\bigcup_{j\in J} A^j_\sigma \mid \sigma \in F(P))$ forms a Komiya cover of $P$. Therefore, $v\in\supp(v) \subset  \bigcup_{\tau \subset \supp(v)} \left(\bigcup_{i\in J} A^i_\tau\right)$. Hence there is some $\tau\subset \supp(v)$ and $i\in J$ such that $v\in A^i_\tau$.
\end{proof}

\section{Refining a triangulation}

Let $X$ be a triangulation of $P$. Denote by $E(X)$ the one-dimensional faces (or {\em edges}) of $X$.  For  an edge $v_1v_2 \in E(X)$,  
let $b(v_1,v_2)$ denote the barycenter of $v_1v_2$. For $m\ge 2$ we define $b(v_1,\dots,v_{m+1})$ recursively, by $$b(v_1,\dots,v_{m+1})=b(b(v_1,\dots,v_{m}),v_{m+1}).$$ Denote by $X(v_1,v_2)$ the triangulation refining $X$, so that $$V(X(v_1,v_2)) = V(X) \cup \{b(v_1,v_2)\},$$ and  $X(v_1,v_2)$ is obtained by replacing every face of the form $\conv\{v_1,v_2,\dots,v_j\}$ in $X$ that contains  the edge $v_1v_2$ with  two faces $\conv\{b(v_1,v_2),v_2,\dots,v_j\}$ and  $\conv\{b(v_1,v_2),v_1,v_3,\dots,v_j\}$.  

Given a map $f:V(X)\to [n]$, 
an edge $uv\in E(X)$  will be called {\em bad} if $f(u)=f(v)$. Let $B(X)$ denote the set of bad edges of $X$, and for a vertex $v\in V(X)$ define $$B(X; v) = \{u\in V(X) \mid uv \in B(X)\}.$$

In this section we prove that given the conditions of Theorem \ref{thm:main}, and given any triangulation $T$ of $P$, there exists a triangulation $T'$ refining $T$ and satisfying certain nice properties.

\begin{theorem}\label{goodtriangulation}
Suppose the conditions of Theorem \ref{thm:main} hold. Let $T$ be a convex triangulation of $P$. 
Then there exists a refinement $T'$ of $T$ and maps $\lambda: V(T')\to F(P)$, $f:V(T')\to [n]$ and $y:V(T') \to P$ with the following properties: 
\begin{enumerate}
    \item[(P1)] $v\in A^{f(v)}_{\lambda(v)}$ for every $v\in V(T')$, 
    \item[(P2)] $y(v) \in \lambda(v) \subset \supp(v)$ for every $v\in V(T')$,  and
    \item[(P3)] for every face $\tau$ of $T'$, the indices $(f(v),~v\in \tau)$ are pairwise distinct.  
\end{enumerate}
\end{theorem}

\begin{proof}
Let $v\in V(T)$.
By Lemma \ref{maychoose} we may choose an index $i \in [n]$ and a face $\tau \subset \supp(v)$ such that $v\in A^i_\tau$. Let $\lambda(v)=\tau$, $f(v)=i$ and  $y(v)=y_\tau^i$.



Let $N=|B(T)|$ be the number of bad edges in $T$. Note that if $N=0$, then properties (P1)-(P3) hold for $T$ with the defined maps $\lambda,f,y$, and we are done. So assume $N\geq 1$, and let $v_1v_2\in B(T)$.

We will now describe an algorithm that receives $T,\lambda,f,y$ and terminates with a refinement $T'$ of $T$ and assignments $\lambda(v) \in F(P)$, $f(v)\in [n]$,  $y(v)\in Y$ for every $v\in V(T')\setminus V(T)$, with the following properties: 
\begin{itemize}
    \item[(i)] properties (P1) and (P2) hold for $T',\lambda,f,y$, 
    \item[(ii)]  $B(T') \subset B(T)$, and
    \item[(iii)] $v_1v_2 \notin E(T')$.
\end{itemize}
Note that (ii) and (iii) implies that $B(T')<B(T)$, and 
 Thus the triangulation $T'$ has at most $N-1$ bad edges. This suffices to prove the theorem.

Throughout the  algorithm, $T_c$ is being updated with finer and finer triangulations. The notation $T_c=T_c(b(v_1,\dots,v_{j+1}),v_{j+2})$ in Step (2) below means that we replace the current $T_c$ with the new refined triangulation $T_c(b(v_1,\dots,v_{j+1}),v_{j+2})$, that we now call $T_c$.
\medskip

{\bf The algorithm. }

Setup:
Set $T_c=T(v_1,v_2)$. Since $|[n]\setminus \{f(v_2)\}|\geq n-k+1$, by Lemma \ref{maychoose}  there exists $i\in [n]\setminus \{f(v_2)\}$ and $\tau \in \supp(b(v_1,v_2))$ so that $b(v_1,v_2)\in A^i_\tau$.
Set $\lambda(b(v_1,v_2))=\tau$, $f(b(v_1,v_2))=i$, and $y(b(v_1,v_2))=y_\tau^i$. Since $y_\tau^i \in \tau \subset \supp(v)$, properties (P1) and (P2) hold for $v=b(v_1,v_2)$.
Set $Q_1=B(T_c;b(v_1,v_2))$  and $Q_j=\emptyset$ for all $j\geq 2$. 

Apply the following procedure:
\begin{enumerate}
    \item If $Q_j=\emptyset$ for all $j\geq 1$,  stop and return $T_c, \lambda, f, y$. Otherwise, 
    \item Let $j$ be the largest index for which $Q_j\neq \emptyset$, and choose a vertex $v\in Q_j$. Set $v_{j+2}=v$,  $T_c=T_c(b(v_1,\dots,v_{j+1}),v_{j+2})$. Remove $v_{j+2}$ from $Q_j$. 
  \item Choose 
    $    i\in [n]\setminus \{f(v_2),f(v_3),\dots,f(v_{j+2})\} $ and $
    \tau \subset \supp(b(v_1,\dots,v_{j+2}))$ such that  $b(v_1,\dots,v_{j+2})\in A^i_\tau$ (we will show that such a choice exists).\\
 Set $\lambda(b(v_1,\dots,v_{j+2}))=\tau$, $f(b(v_1,\dots,v_{j+2}))=i$,  $y({b(v_1,\dots,v_{j+2})})=y_{\tau}^i$. Like before, properties (P1) and (P2) hold for $v=b(v_1,\dots,v_{j+2})$.
    \item Set 
  $Q_{j+1}=B(T_c;b(v_1,\dots,v_{j+2})).$
\item Return to Step (1).
\end{enumerate}


\medskip
The idea of the algorithm is as follows. We want to eliminate the bad edge $v_1v_2$ of $T$ by subdividing it using the new vertex $b(v_1,v_2)$ and refining $T$ to obtain $T(v_1,v_2)$. However, after assigning $f(b(v_1,v_2))$, the current triangulation $T(v_1,v_2)$ may contain bad edges that did not appear in $T$. Every such a bad edge must contain the vertex $b(v_1,v_2)$. So we let $Q_1$ record the vertices $v$ for which $vb(v_1,v_2)$ is a bad edge; these are precisely the bad edges of $T(v_1,v_2)$ that are not bad in $T$. 
More generally, throughout the algorithm  $Q_j$ records the vertices $v$ of the current triangulation $T_c$ such that $vb(v_1,\dots,v_{j+1})$ is a bad edge of $T_c$. At every iteration of the algorithm, every bad edge of $T_c$ that is not bad in $T$ is of the form $vb(v_1,\dots,v_{j+1})$, for some $j\ge 1$. Thus, if the algorithm terminates, i.e. $Q_j=\emptyset$ for all $j$, then the returning triangulation $T_c$ contains only bad edges that were already bad in $T$. Moreover, $T_c$ does not contain the edge $v_1v_2$, which was bad in $T$. Thus, $T_c$  contains at least one less bad edge than $T$.

The algorithm works in a depth-first manner. We find the largest index $j$ such that $Q_j\neq \emptyset$ and we reduce the size of $Q_j$ by subdividing an edge of the form $vb(v_1,\dots,v_{j+1})$. By doing this we may have $Q_{j+1}\neq \emptyset$  in Step (4). The algorithm continues to run until $Q_{j+1}=\emptyset$ (we will show that this always happens), and then we reduce the size of $Q_j$ by subdividing another edge of the form $vb(v_1,\dots,v_{j+1})$. Similarly, we continue until $Q_{j}=\emptyset$, and then we reduce the size of $Q_{j-1}$ by subdividing an edge of the form $vb(v_1,\dots,v_{j})$  (see example \ref{exalgorithm}).
\medskip

Our goal now is to show that: (a) the choice in Step (3) of the algorithm is possible as long as the algorithm runs,  (b) the algorithm stops after finitely many steps, and (c) the returning triangulation satisfies Properties (i)-(iii). This will imply the theorem. 

\begin{claim}
For every $2\leq j\leq k$, each maximal simplex of $T_c$ containing the vertex $b(v_1,\dots,v_{j})$ is of the form 
$$
\conv\{b(v_1,\dots,v_{j}), w_2, \dots, w_{j},u_{j+1},\dots,u_k\},
$$
where 
\begin{itemize}
    \item $w_2\in\{v_1, v_2\}$, 
    \item $w_i \in \{v_i, b(v_1,\dots,v_{i-1})\}$ for every $3\leq i\leq j$,  and
    \item $u_{j+1},\dots,u_k$ are some vertices in $V(T_c)$.
\end{itemize}
 \end{claim}
 \begin{proof}
We proceed by induction on $j$. For $j=2$ the claim is true. Indeed, in the triangulation $T(v_1,v_2)$ each maximal simplex containing $b(v_1,v_2)$ contains $v_1$ or $v_2$. Since $f(b(v_1,v_2))\neq f(v_1), f(v_2)$, the edges $v_1b(v_1,v_2)$ or $v_2b(v_1,v_2)$ are never subdivided at any iteration in the algorithm, so in each triangulation obtained at any iteration of the algorithm, any maximal simplex containing $b(v_1,v_2)$ will still contain $v_1$ or $v_2$. 
Let  $3\leq j\leq k$, and assume the claim is true for $j-1$. 

We first claim that $v_{j}$ is not one of the vertices $v_1,\dots,v_{j-1}$ or $b(v_1,\dots,v_{i})$ for $2\leq i\leq j-1$. This is true because at some iteration in the algorithm, $v_j\in  B(b(v_1,\dots,v_{j-1}))$ and thus $v_jb(v_1,\dots,v_{j-1})$ is a bad edge (and in particular $v_j\neq b(v_1,\dots,v_{j-1})$), and $ub(v_1,\dots,v_{j-1})$ is not a bad edge if $u$ is one of  the vertices $v_1,\dots,v_{j-1}$ or $b(v_1,\dots,v_{i})$ for $2\leq i\leq j-2$; indeed, we have that $f(v_1)=f(v_2)$ and $f(v_{i+1})=f(b(v_1,\dots,v_{i}))$ for $2\leq i\leq j-2$, and by the choice of $f(b(v_1,\dots,v_{i}))$ made in Step (3), $f(b(v_1,\dots,v_{i}))\in [n]\setminus \{f(v_2),f(v_3),\dots,f(v_i)\}$.

By the inductive hypothesis, the maximal simplices of $T_c$ containing the edge $v_{j}b(v_1,\dots,v_{j-1})$ are of the form
\[
\conv\{b(v_1,\dots,v_{j-1}), w_2, \dots, w_{j-1},v_{j},u_{j+1},\dots,u_k\}
\]
where $w_2\in \{v_1, v_2\}$ and for $3\leq i\leq j$, $w_i$ is  equal to $v_i$ or $b(v_1,\dots,v_{i-1})$. Therefore, the maximal simplices of $T_c$ containing $b(v_1,\dots,v_{j})$ are of the form
\[
\conv\{b(v_1,\dots,v_{j}), w_2, \dots, w_{j},u_{j+1},\dots,u_k\}
\]
where $w_2\in \{v_1, v_2\}$ and for $3\leq i\leq j$, $w_i$ is  equal to $v_i$ or $b(v_1,\dots,v_{i-1})$. 
This completes the proof of the claim.
\end{proof}

Applying the claim for $j=k$ we get:

\begin{claim}
\label{b^{k-1}}
Assume that at some iteration of the algorithm, at Step (2) we have $Q_{k-2}\neq \emptyset$ and $k-2$ is the largest index $j$ for which $Q_j\neq \emptyset$. Then in the current triangulation $T_c$, the vertex $b(v_1,\dots,v_k)$ is connected by an edge only to the vertices $v_1,\dots,v_k$ and   $b(v_1,\dots,v_{i-1})$ for $3\leq i\leq k$.
\end{claim}

\begin{claim}\label{Q_k-1}
At any iteration of the algorithm,  $Q_{j}=\emptyset$ for every $j\ge k-1$. In other words, the index $j$ from Step (2) is always at most $k-2$. 
\end{claim}
\begin{proof}
We first prove that $Q_{k-1}=\emptyset$ at every iteration. 
Assume that $v_k$ has been defined in some iteration of the algorithm  in Step (2), and let $T_c$ be the resulting triangulation defined in the same step. By Claim \ref{b^{k-1}}, $b(v_1,\dots,v_k)$ is connected by an edge only to the vertices $v_1,\dots,v_k$ and $b(v_1,\dots,v_{j})$ for $2\leq j\leq k-1$ in $T_c$. By the choice of $f(b(v_1,\dots,v_k))$ from Step (3), we have  $$f(b(v_1,\dots,v_k))\in [n]\setminus \{f(v_2),f(v_3),\dots,f(v_k)\}.$$ Moreover, $f(v_1)=f(v_2)$ and by the setup in Step (2)  $f(b(v_1,\dots,v_{i}))=f(v_{i+1})$ for every $2\leq i \leq k-1$. Therefore, $f(b(v_1,\dots,v_k))\neq f(v)$ for every vertex $v$ that is connected to $b(v_1,\dots,v_k)$ by an edge in $T_c$. It follows that $Q_{k-1}=\emptyset$.

Now, since $Q_{j+1}$ is being changed in Step (4) only if $Q_j\neq \emptyset$ at Step (2) in some iteration of the algorithm, it follows that  $Q_j= \emptyset$ for every $j\ge k$ as well.
\end{proof}

\begin{claim}
The choice of $i,\tau$ in Step (3) of the algorithm is possible as long as the algorithm runs. 
\end{claim}
\begin{proof}
By Claim \ref{Q_k-1}, at any iteration of the algorithm $j\le k-2$, and thus the set $I=[n]\setminus \{f(v_2),f(v_3),\dots,f(v_{j+2})\}$
is of size at least $n-k+1$. Therefore, by Lemma \ref{maychoose} there exists a choice of $i\in I$ and $\tau \in \supp(b(v_1\dots,v_{j+2}))$ so that $b(v_1\dots,v_{j+2})\in A^i_\tau$, as needed in Step (3).
\end{proof}

\begin{claim}\label{Qj empty}
For any $j\ge 1$, at any iteration of the algorithm, either $Q_j=\emptyset$ or there is a later iteration in the algorithm for which  $Q_j = \emptyset$ and the size of $Q_{i}$ for each $1\leq i\leq j-1$ is the same in both iterations.
\end{claim}

\begin{proof}
We proceed by induction on $k-1-j$. 
If $k-1-j\le 0$, then the statement is true by Claim \ref{Q_k-1}. 

Let $1\leq j\leq k-2$.
Assume that in the $i_1$-th iteration in the algorithm, we have  $Q_j\neq \emptyset$. It suffices to show that there is some later iteration for which the size of $Q_j$ is decreased by $1$ and the size of $Q_i$ for $1\leq i\leq j-1$ remains unchanged. By the induction hypothesis, there exists $i_2\ge i_1$  such that in the $i_2$-th iteration of the algorithm, $Q_{j+1}=\emptyset$ and the size of $Q_i$ for each $1\leq i\leq j$ is the same as in the $i_1$-th iteration. 
Similarly, we may find $i_3\ge i_2$ such that in the $i_3$-th iteration of the algorithm, $Q_{j+2}=\emptyset$ and the size of $Q_i$, for every $1\leq i\leq j+1$, is the same as in the $i_2$-th iteration (and in particular,  $Q_{j+1}=\emptyset$). Continuing in this way, we find some $i_{k-1-j}\ge i_1$  such that in the $i_{k-1-j}$-th iteration of the algorithm, $Q_i=\emptyset$ for all $j+1\leq i\leq k-2$ and the size of $Q_i$ for $1\leq i\leq j$ is the same as in the $i_1$-th iteration. It follows by Claim \ref{Q_k-1} that $j$ is the largest index for which $Q_j\neq \emptyset$
Therefore, in Step (2) of the $i_{k-j}$-th iteration, we choose some $v\in Q_j$ and set $Q_j=Q_j\setminus \{v\}$, so the size of $Q_j$ deceases by $1$, and the size of $Q_i$ for $1\leq i\leq j-1$ is unchanged. 
\end{proof}

\begin{claim}\label{emptylater}
If at some iteration of the algorithm $Q_i=\emptyset$ for evey $1\leq i\leq j$, then $Q_j=\emptyset$ in every later iteration of the algorithm.
\end{claim}
\begin{proof}
If $Q_i=\emptyset$ for every $1\leq i\leq j$, then at the next iteration of the algorithm, the index chosen in Step (2) is at least $j+1$. Therefore, it is still the case that $Q_i=\emptyset$ for each $1\leq i\leq j$ in the next iteration. 
\end{proof}

\begin{claim}\label{BadEdges}
At any iteration of the algorithm, every edge of $B(T_c)\setminus B(T)$ is of the form $vb(v_1,\dots,v_j)$, where  $v\in Q_{j-1}$ and $v_j$ has been defined in some iteration of algorithm.
\end{claim}
\begin{proof}
We proceed by induction on the number of iterations. The claim is  true in the $0$-iteration, when $T_c=T(v_1,v_2)$  in the setup. 

Let $T'$ be the triangulation $T_c$ obtained at the $i$-ith iteration of the algorithm for some $i$, and assume the statement holds for $T'$. Let $T''$ be the triangulation obtained in the $(i+1)$-th iteration. Let $j$ be the index in Step (2) of the $(i+1)$-th iteration. Then every bad edge of $E(T'')\setminus E(T')$ is of the form $vb(v_1,\dots,v_{j+2})$, for $v\in Q_{j+1}$. This, combined with the fact that the claim held for $T'$, implies that the claim holds for $T''$. 
\end{proof}

\begin{claim}\label{terminates}
The algorithm terminates after finitely many steps. The returning triangulation $T'$ satisfies properties (i),(ii) and (iii). 
\end{claim}
\begin{proof}
Applying Claim \ref{Qj empty} for $j=1$, we have that there is some iteration for which $Q_1=\emptyset$. Therefore, by Claim \ref{emptylater}, we have that $Q_1=\emptyset$ at  any later iteration of the algorithm. Now, by Claim \ref{Qj empty}, there is a later iteration in the algorithm for which $Q_1=Q_2=\emptyset$. Applying Claim \ref{emptylater} again, we have that $Q_2=\emptyset$ for every later iteration. Continuing  this way, we get that there is an iteration where $Q_j=\emptyset$ for every $j\geq 1$, and therefore the algorithm  terminates.

By Claim \ref{BadEdges}, every edge in  $B(T')\setminus B(T)$ is of the form $vb(v_1,\dots,v_{j+1})$ where $v\in Q_{j}$. When the algorithm terminates, we have that $Q_j=\emptyset$ for all $j\geq 1$, so $T'$ has no such bad edges. Since $T'$ no longer has the edge $v_1v_2$, it follows that $T'$ has at most $N-1$ bad edges. Finally, (i) holds by the the definition of $\lambda, f, y$ throughout the algorithm.
\end{proof}

Claim \ref{terminates} completes the proof of Theorem \ref{goodtriangulation}.
\end{proof}

\begin{example}\label{exalgorithm}
Let $P$ be a 2-dimensional simplex and $n=4$. Assume we have closed sets $(A^i_\sigma \mid i\in [4], \sigma \in F(P))$  satisfying the conditions of Theorem \ref{thm:main}. Let $T$ be  the triangulation of $P$ depicted in Figure \ref{fig:ex}A. Suppose that a map $f:V(T)\to [4]$ is defined.  Every vertex $v\in V(T)$ in the figure is labeled by $v,f(v)$. Then 
 $u_1u_2$ is a bad edge.

 We now apply the algorithm with $v_1=u_1$ and $v_2=u_2$. In the setup of the algorithm we obtain  $T_c=T(u_1,u_2)$, and assign $f(b(v_1,v_2))\in [4]\setminus\{1\}$. Say $f(b(v_1,v_2))=2$. We then obtain the triangulation in  Figure \ref{fig:ex}B.

Proceeding to Step (1) in the algorithm, we find that $Q_1=\{u_3,u_4\}$ is non-empty, and in Step (2), the largest index $j$ such that $Q_j\neq \emptyset$ is $j=1$. We choose $v_3=u_4$,  $T_c=T_c(b(v_1,v_2),v_3)$, and remove $v_3$ from $Q_1$, so now $Q_1=\{u_3\}$.
Proceeding now to Step (3), we find some $i\in [4]\setminus \{f(v_2)=1,f(v_3)=2\}$ and $\sigma\subset \supp(b(v_1,v_2,v_3))$ such that $b(v_1,v_2,v_3)\in A^i_\sigma$, and set $f(b(v_1,v_2,v_3))=i$. Say  $i=3$. 
We obtain the triangulation in  Figure \ref{fig:ex}C.
In step (4) we find $Q_2=B(T_c, b(u_1,u_2,u_4)) = \emptyset$.

We go back to Step (1) and begin the second iteration of the algorithm. Since $Q_1\neq \emptyset$, we move to Step (2).  Here $j=1$ is again the largest index for which $Q_j\neq \emptyset$. We set $v_3=u_3$, $T_c=T_c(b(v_1,v_2),v_3)$, and $Q_1=Q_1\setminus \{v_3\}=\emptyset$. In Step (3), we find some $i\in [4]\setminus \{f(v_2)=1,f(v_3)=2\}$ and $\sigma \subset \supp(b(v_1,v_2,v_3))$ such that $b(v_1,v_2,v_3)\in A^i_\sigma$. We set $f(b(v_1,v_2,v_3)=i$. Say  $i=4$. We obtain the triangulation $T_c$ in  Figure \ref{fig:ex}D.
In Step (4), $Q_2=\emptyset$.

Now we go back to Step (1). The algorithm terminates since $Q_j=\emptyset$ for all $j$.
In the returning triangulation $T'$ every bad edge was bad already in the triangulation $T$ we started with. Moreover, the edge $u_1u_2$ that was bad  in $T$, is not an edge anymore in $T'$. 

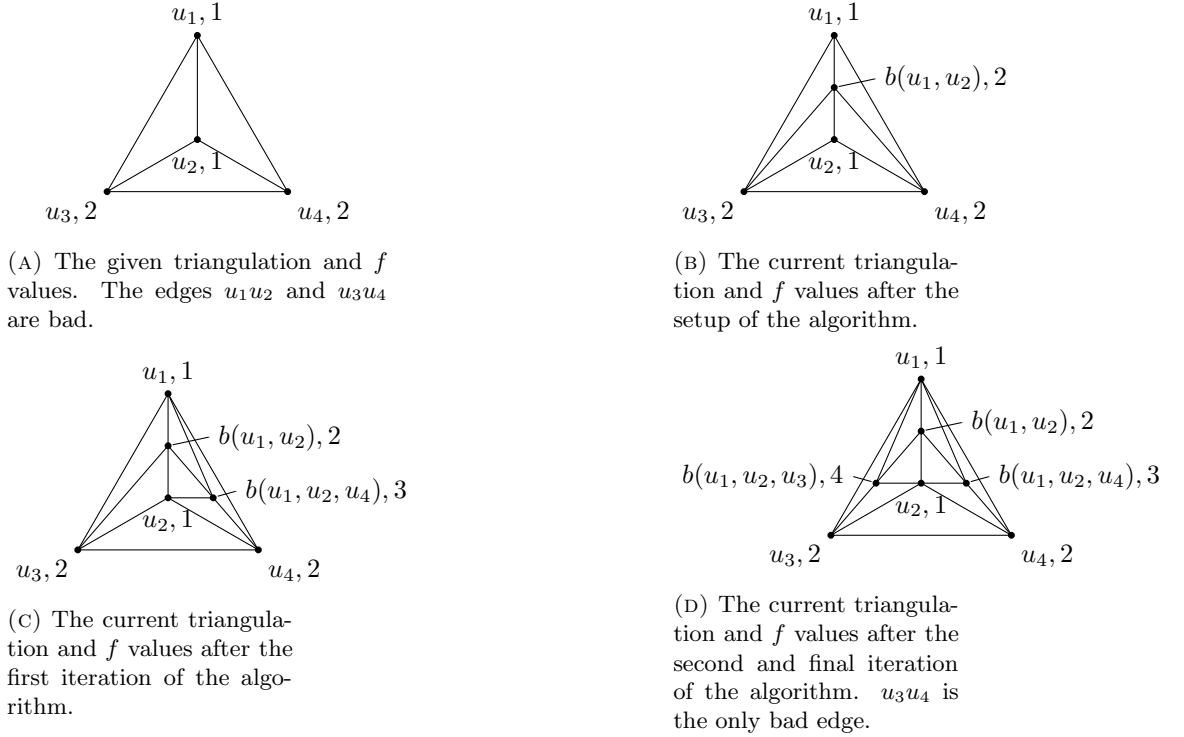
\begin{figure}[h]
    \centering
 \begin{subfigure}{0.4\textwidth}
         \centering
    \begin{tikzpicture}[scale=.6]
    \coordinate (a) at (-2,0);
    \coordinate (b) at (2,0);
    \coordinate (c) at (0,3.464);
    \coordinate (d) at (0,1.155);
     \fill[black, draw=black, thick] (a) circle (1.5pt) node[black, below left] {$u_3,2$};
    \fill[black, draw=black, thick] (b) circle (1.5pt) node[black, below right] {$u_4,2$};
    \fill[black, draw=black, thick] (c) circle (1.5pt) node[black, above] {$u_1,1$};
    \fill[black, draw=black, thick] (d) circle (1.5pt) node[black, below] {};
    \fill[black, draw=black, thick] (0,1.1) circle (0pt) node[black, below] {$u_2,1$};
    \draw (a) -- (b);
    \draw (a) -- (c);
    \draw (c) -- (b);
    \draw (a) -- (d);
    \draw (d) -- (c);
    \draw (d) -- (b);
    \end{tikzpicture}
        \caption{The given triangulation and $f$ values. The edges $u_1u_2$ and $u_3u_4$ are bad.}\label{1A}
     \end{subfigure}
     \hfill
         \begin{subfigure}{0.3\textwidth}
         \centering
    \begin{tikzpicture}[scale=.6]
    \coordinate (a) at (-2,0);
    \coordinate (b) at (2,0);
    \coordinate (c) at (0,3.464);
    \coordinate (d) at (0,1.155);
    \coordinate (e) at (0,2.3095);
    
    \fill[black, draw=black, thick] (a) circle (1.5pt) node[black, below left] {$u_3,2$};
    \fill[black, draw=black, thick] (b) circle (1.5pt) node[black, below right] {$u_4,2$};
    \fill[black, draw=black, thick] (c) circle (1.5pt) node[black, above] {$u_1,1$};
    \fill[black, draw=black, thick] (d) circle (1.5pt) node[black, below] {};
    \fill[black, draw=black, thick] (0,1.1) circle (0pt) node[black, below] {$u_2,1$};
    \fill[black, draw=black, thick] (e) circle (1.5pt) node[black, above right] {};
    \fill[black, draw=black, thick] (.9,2.5) circle (0pt) node[black, right] {$b(u_1,u_2),2$};
    
    \draw (.9,2.5) -- (0.1,2.33);
    
    \draw (a) -- (b);
    \draw (a) -- (c);
    \draw (c) -- (b);
    
    \draw (a) -- (d);
    \draw (d) -- (c);
    \draw (d) -- (b);
    
    \draw (e) -- (a);
    \draw (e) -- (b);
    \end{tikzpicture}
     \caption{The current triangulation and $f$ values after the setup of the algorithm.}
     \end{subfigure}
         \begin{subfigure}{0.3\textwidth}
    \begin{tikzpicture}[scale=.6]
    \coordinate (a) at (-2,0);
    \coordinate (b) at (2,0);
    \coordinate (c) at (0,3.464);
    \coordinate (d) at (0,1.155);
    \coordinate (e) at (0,2.3095);
    \coordinate (f) at (1,1.155);
    
    \fill[black, draw=black, thick] (a) circle (1.5pt) node[black, below left] {$u_3,2$};
    \fill[black, draw=black, thick] (b) circle (1.5pt) node[black, below right] {$u_4,2$};
    \fill[black, draw=black, thick] (c) circle (1.5pt) node[black, above] {$u_1,1$};
    \fill[black, draw=black, thick] (d) circle (1.5pt) node[black, below] {};
    \fill[black, draw=black, thick] (0,1.1) circle (0pt) node[black, below] {$u_2,1$};
    \fill[black, draw=black, thick] (e) circle (1.5pt) node[black, above right] {};
    \fill[black, draw=black, thick] (.9,2.5) circle (0pt) node[black, right] {$b(u_1,u_2),2$};
    \fill[black, draw=black, thick] (f) circle (1.5pt) node[black, above right] {};
    \fill[black, draw=black, thick] (1.5,1.3) circle (0pt) node[black, right] {$b(u_1,u_2,u_4),3$};
    
    \draw (.9,2.5) -- (0.1,2.33);
    \draw (1.5,1.3) -- (1.1,1.19);

    \draw (a) -- (b);
    \draw (a) -- (c);
    \draw (c) -- (b);
    
    \draw (a) -- (d);
    \draw (d) -- (c);
    \draw (d) -- (b);
    
    \draw (e) -- (a);
    \draw (e) -- (b);
    
    \draw (f) -- (c);
    \draw (f) -- (d);
    \end{tikzpicture}
    \caption{The current triangulation and $f$ values after the first iteration of the algorithm.}
     \end{subfigure}
     \hfill
         \begin{subfigure}{0.3\textwidth}
         \centering
    \begin{tikzpicture}[scale=.6]
    \coordinate (a) at (-2,0);
    \coordinate (b) at (2,0);
    \coordinate (c) at (0,3.464);
    \coordinate (d) at (0,1.155);
    \coordinate (e) at (0,2.3095);
    \coordinate (f) at (1,1.155);
    \coordinate (g) at (-1,1.155);
    
    \fill[black, draw=black, thick] (a) circle (1.5pt) node[black, below left] {$u_3,2$};
    \fill[black, draw=black, thick] (b) circle (1.5pt) node[black, below right] {$u_4,2$};
    \fill[black, draw=black, thick] (c) circle (1.5pt) node[black, above] {$u_1,1$};
    \fill[black, draw=black, thick] (d) circle (1.5pt) node[black, below] {};
    \fill[black, draw=black, thick] (0,1.1) circle (0pt) node[black, below] {$u_2,1$};
    \fill[black, draw=black, thick] (e) circle (1.5pt) node[black, above right] {};
    \fill[black, draw=black, thick] (.9,2.5) circle (0pt) node[black, right] {$b(u_1,u_2),2$};
    \fill[black, draw=black, thick] (f) circle (1.5pt) node[black, above right] {};
    \fill[black, draw=black, thick] (1.5,1.3) circle (0pt) node[black, right] {$b(u_1,u_2,u_4),3$};
    \fill[black, draw=black, thick] (g) circle (1.5pt) node[black, left] {};
    \fill[black, draw=black, thick] (-1.5,1.3) circle (0pt) node[black, left] {$b(u_1,u_2,u_3),4$};
    
    \draw (.9,2.5) -- (0.1,2.33);
    \draw (1.5,1.3) -- (1.1,1.19);
    \draw (-1.5,1.3) -- (-1.1,1.19);

    \draw (a) -- (b);
    \draw (a) -- (c);
    \draw (c) -- (b);
    
    \draw (a) -- (d);
    \draw (d) -- (c);
    \draw (d) -- (b);
    
    \draw (e) -- (a);
    \draw (e) -- (b);
    
    \draw (f) -- (c);
    \draw (f) -- (d);
    
    \draw (g) -- (c);
    \draw (g) -- (d);
    \end{tikzpicture}
    \caption{The current triangulation and $f$ values after the second and final iteration of the algorithm. $u_3u_4$ is the only bad edge.}
         \end{subfigure}
  
\caption{The algorithm applied to a triangulation of a 2-simplex in Example \ref{exalgorithm}.}
\label{fig:ex}
\end{figure}

\end{example}

\section{Proof of Theorem \ref{thm:main}}
The proof follows  from Theorem \ref{goodtriangulation} using an argument similar to the one used to prove Theorem 2.1 in   \cite{floriancolorful2019}.
By Theorem \ref{goodtriangulation}, for every  $\varepsilon >0$
there exists a  triangulation $T_\varepsilon$ of $P$ with diameter at most $\varepsilon$ and assignments $\lambda(v)$, $f(v)$,  $y(v)$,  such that the following properties hold:
\begin{enumerate}
    \item for every $v\in V(T_\varepsilon)$ we have $v\in A_{\lambda(v)}^{f(v)}$
    and $y(v)=y_{\lambda(v)}^{f(v)}$,
    \item $\lambda(v)\subset \supp(v)$, and
    \item for every face $\tau$ of $T_\epsilon$, the indices $(f(u)\mid u\in V(\tau))$ are pairwise distinct.
\end{enumerate}

 By Theorem \ref{sperner-shapley} there is a face $\sigma=\conv\{u_1,\dots,u_k\}$ in $T_\varepsilon$ (of dimension $k-1$, without loss of generality) 
 such that 
 $
 p\in  \conv\{y(u) \mid  u\in V(\sigma)\}.$
 This implies that 
 for $\tau_i=\lambda(u_i)$ and $j_i=f(u_i)$, we have $p\in  \conv\{y_{\tau_i}^{j_i} \mid 1\leq i\leq k\}.$ Moreover, since $u_i \in A_{\sigma_1}^{j_1}$ and the diameter of $\sigma$ is at most $\varepsilon$, the $\varepsilon$-neighborhoods of $A_{\sigma_1}^{j_1},\dots, A_{\sigma_k}^{j_k}$ intersect. Let $\pi_\varepsilon(i): [k]\to [n]$ be the injection $i\mapsto j_i$.
 
 Now, taking  $\varepsilon$ to $0$ and using the compactness of $P$ and the fact that the sets $A_{\tau}^{i}$ are closed, we conclude the theorem. \qed

\section{Applications: proofs of Theorems \ref{thm:cakes} and \ref{thm:sparse colorful-intervals}}

The proof of Theorem \ref{thm:sparse colorful-intervals} is very similar to the proof  of Theorem \ref{thm:colorful-intervals}, where the role of Theorem \ref{frick-zerbib} is replaced by Theorem \ref{thm:main}. 

A {\em fractional matching} in a hypergraph $H=(V,E)$ is a function $f\colon E \to \mathbb{R}_{\ge 0}$ satisfying $\sum_{e:~ e\ni v} f(e)\le 1$ 
for all~${v\in V}$.  The \emph{fractional matching number} $\nu^*(H)$ is the maximum of $\sum_{e\in E} f(e)$ over all fractional matchings $f$ of $H$.
A {\em perfect fractional matching} in $H$ is a fractional matching $f$ in which $\sum_{e:v\in e} f(e) = 1$ for every $v\in V$. 
The {\em rank} of a hypergraph $H=(V,E)$ is the maximal size of an edge in $H$. A hypergraph $H$ is $d$-partite if there exists a partition $V_1,\dots, V_d$ of $V$ such that $|e\cap V_i| =1$ for every $e\in E$ and $i\in [d]$. 

For the proof of Theorem \ref{thm:sparse colorful-intervals} we will use a  theorem by F\"uredi \cite{furedi}.
\begin{theorem}[F\"uredi \cite{furedi}]\label{furedi}
If $H$ is a hypergraph of rank $d\ge 2$, then
$\nu(H) \ge \frac{\nu^*(H)}{d-1+\frac{1}{d}}.$ If $H$ is $d$-partite, then $\nu(H) \ge \frac{\nu^*(H)}{d-1}.$  
\end{theorem}

We will also need the following simple  lemma (see e.g. \cite{floriancolorful2019}). 
\begin{lemma}\label{rank}
If a hypergraph $H=(V,E)$ of rank $d$ has a perfect fractional matching, then $\nu^*(H)\ge \frac{|V|}{d}$.  
\end{lemma}

\begin{proof}[Proof of Theorem \ref{thm:sparse colorful-intervals}]
For a point 
$x=(x_1,\dots,x_k) \in \Delta^{k-1}$ let  
$p_x(j)=\sum_{i=1}^j x_i \in [0,1]$.
Since $\F=\bigcup_{i=1}^n \F_i$ is finite, by rescaling $\mathbb{R}$ we may assume that each member of
$\F$ is a subset of $(0,1)$. 
Every face of $\Delta^{k-1}$ is of the form $\Delta^T=\conv \{e_j \mid j\in T\}$ for some $T\subset [k]$, where $e_j$ is the vertex of $\Delta^{k-1}$  having 1 at the $j$-component and 0 otherwise. 

For every $T\subset [k]$, let $A^i_T$ be the set consisting of all $x \in
\Delta^{k-1}$ for which there exists a $d$-interval $f \in \F_i $ satisfying: 
\begin{enumerate}
\item[(a)] $f\subset \bigcup_{j\in T} (p_{x}({j-1}),p_{x}({j}))$, and 
\item[(b)] $f \cap
(p_{x}({j-1}),p_{x}({j})) \neq \emptyset$ for each $j
\in T$.
\end{enumerate}  
Note that $A^i_T =
\emptyset$ whenever $|T| > d$, 
and that the sets $A^i_{T}$ are open.

Let $I\in {[n] \choose n-k+1}$.  The assumption $\tau(\bigcup_{i\in I} F_i )\ge k$ implies that for every $x=(x_1,\dots,x_{k}) \in \Delta^{k-1}$, the set $P(x)=\{p_{x}({j}) \::\: j\in [k-1]\}$ is not a cover of $\bigcup_{i\in I} F_i$,  
meaning that there exists $i\in I$ and $f\in \F_i $ not
containing any $p_{x}(j)$. This, in turn, means that $x
\in A^i_{T}$ for some $T \subset [k]$. Thus the  sets $(A^i_T \mid i\in I, T\subset [k])$  form a cover of $\Delta^{k-1}$.
  To show that this is a $k$-weakly Komiya cover, let $\Delta^S$ be a face of $\Delta^{k-1}$ for some $S\subset [k]$. If $x\in \Delta^S$ then
$(p_{x}({j-1}),p_{x}({j}))=\emptyset$ for $j \notin S$,
and hence it is impossible to have ${f \cap (p_{x}({j-1}),p_{x}({j}))\neq \emptyset}$. 
Thus  $x \in A^i_{T}$
for some $T \subseteq S$. This proves that $\Delta^S \subseteq \bigcup_{T
\subseteq S} \bigcup_{i\in I} A^i_{T}$, and thus $(\bigcup_{i\in I} A^i_{T} \mid T\subset [k])$ is a Komiya cover, for every $I\in {[n] \choose n-k+1}$.

By Theorem \ref{thm:main} 
there is an injection $\pi: [k]\rightarrow [n]$ and faces $\Delta^{T_1},\dots,\Delta^{T_k}$ of $\Delta^{k-1}$ such that 
\begin{enumerate}
    \item  $b(\Delta^{k-1}) \in  \conv\{b(\Delta^{T_1}),\dots,b(\Delta^{T_k})\}$, and
    \item  $\bigcap_{i=1}^k A^{\pi(i)}_{T_{i}}\neq \emptyset$. 
\end{enumerate}

Note that (1) implies that 
 the hypergraph $H=([k], \{T_{1},\dots,T_{k}\})$ has a perfect fractional matching, and (2) implies that 
$|T_i|\le d$ for all $i \in [k]$. 

Then by the observation mentioned above $\nu^*(H) \ge
\frac{k}{d}$. Therefore, by Theorem \ref{furedi}, $\nu(H) \ge
\frac{\nu^*(H)}{d-1+\frac{1}{d}}\ge \frac{k}{d^2-d+1}.$ 

Let $M$ be a matching in $H$ of size  $m \ge \frac{k}{d^2-d+1}$.
Let $x \in \bigcap_{i=1}^{k} A^{\pi(i)}_{T_i}$. For every $i\in [k]$ let $f_i$ be the $d$-interval of $\F_{\pi(i)} $ witnessing the fact that
$x \in A^{\pi(i)}_{T_i}$. Then the set $\M=\{f_i\mid T_i \in M\}$ is a colorful matching of
size $m$ in $\F$. This proves the first assertion of the theorem.

Now suppose that $\F_i $ is a hypergraph of separated $d$-intervals for all $i\in [n]$. 
For $f \in \F$ let $f^t \subset (t-1,t)$ be the $t$-th interval component of~$f$.
We can assume without loss of generality that $f^t$ is nonempty.  
Let $P=(\Delta^{m-1})^d$. Then $\dim P = (m-1)d$. For a $d$-tuple $T=(j_i,\dots,j_d) \in [m]^d$ (note that $T$ corresponds to the vertex $v_T=e_{j_1}\times e_{j_2} \times \cdots \times e_{j_d}$ of $P$) let $A^i_T$ consist of all $x=x^1\times\cdots\times x^d \in
P$ for which there exists $f \in \F_i $ satisfying 
 $f^t\subset (t-1+p_{x^t}({j_t-1}),t-1+p_{x^t}({j_t}))$ for all $t\in[d]$.

Let $I\in {[n] \choose n-d(m-1)}$. Since $\tau(\bigcup_{i\in I} \F_{i}) \ge d(m-1)+1$, the set of points $$\{t-1+p_{x^t}(j)\mid t\in[d], j\in [m-1]\}$$ do not form a cover of $\bigcup_{i\in I} \F_{i}$. Therefore,  
by the same argument as before, the sets $(A^i_{T} \mid i\in[n], T\in [m]^d) $ are open and form a $(d(m-1)+1)$-weakly Komiya cover of $P$.

Applying Theorem \ref{thm:main} with $k=d(m-1)+1$,
 we conclude that there exists an injection $\pi:[d(m-1)+1]\to [n]$ and $d$-tuples
 $T_1,\dots,T_{d(m-1)+1} $ in $[m]^d$ such that 
 \begin{enumerate}
     \item $b(P) = \conv\{v_{T_1},\dots,v_{T_k}\} $, and
     \item $\bigcap_{i\in [(m-1)d+1]} A^{\pi(i)}_{T_i} \neq \emptyset$.
 \end{enumerate}
Observe that (1) implies that the $d$-partite hypergraph  $$([m]\times [d] ,\{T_1,\dots,T_{d(m-1)+1} \})$$ (where an edge $T_i$ contains a vertex $(i,j)$ if $T_i$ has $i$  in its $j$-th entry) has a  perfect fractional matching. Hence by Lemma \ref{rank} we have $\nu^*(H) \ge md/d =m$. By Theorem \ref{furedi}, this implies $\nu(H) \ge
\frac{\nu^*(H)}{d-1}\ge \frac{m}{d-1}
.$ Now, by the same argument as before, by taking $x \in \bigcap_{i\in [(m-1)d+1]} A^{\pi(i)}_{T_i}$ we obtain a colorful matching in $\F$ of the same size, concluding the proof of the theorem.
\end{proof}

\begin{proof}[Proof of Theorem \ref{thm:cakes}]
Given $d$ cakes, the set of all possible partitions of the cakes in $m$ interval pieces is modeled by the polytope  
$P=(\Delta^{m-1})^d$  of dimension  $(m-1)d$ (see e.g. \cite{NSZ, ABBSZ}). For a $d$-tuple $T=(j_i,\dots,j_d) \in [m]^d$ define $$A^i_T=\{x=x^1\times\cdots\times x^d \in
P \mid \text{Player } i \text{ prefers the } d\text{-tuple of pieces }T\}.$$ The fact that the players are $(m,d)$ hungry implies that the sets $(A^i_T \mid i\in [n], T\in [m]^d)$ form a $(d(m-1)+1)$-weakly Komiya cover of $P$.
Applying Theorem \ref{thm:main} with $k=d(m-1)+1$ as in the previous proof,
 we conclude that there exists an injection $\pi:[d(m-1)+1]\to [n]$ and $d$-tuples
 $T_1,\dots,T_{d(m-1)+1} $   in $[m]^d$ such that 
 \begin{enumerate}
     \item $b(P) = \conv\{v_{T_1},\dots,v_{T_k}\} $, and
     \item $\bigcap_{i\in [(m-1)d+1]} A^{\pi(i)}_{T_i} \neq \emptyset$.
 \end{enumerate}

Like before, (1) implies that the $d$-partite hypergraph  $$H=([m]\times [d] ,\{T_1,\dots,T_{d(m-1)+1} \})$$ has a  perfect fractional matching. Hence by Lemma \ref{rank} we have $\nu^*(H) \ge md/d =m$. By Theorem \ref{furedi}, this implies $\nu(H) \ge
\frac{\nu^*(H)}{d-1}\ge \frac{m}{d-1}
.$ 

Let $M$ be a matching of size at least $\frac{m}{d-1}$ in $H$.
Consider the partition $x \in \bigcap_{i\in [(m-1)d+1]} A^{\pi(i)}_{T_i}$. Then players $\{\pi(i) \mid T_i\in M\}$ prefer in the partition $x$ pairwise disjoint d-tuples of pieces, as needed.   
\end{proof}


\begin{thebibliography}{10}

\bibitem{ABBSZ}
R.~Aharoni, E.~Berger, J.~Briggs, E.~Segal-Halevi, and S.~Zerbib.
\newblock Fractionally balanced hypergraphs and rainbow kkm theorems.
\newblock {\em Combinatorica}, to appear.

\bibitem{AharoniFractional2017}
R.~Aharoni, T.~Kaiser, and S.~Zerbib.
\newblock Fractional covers and matchings in families of weighted
  {$d$}-intervals.
\newblock {\em Combinatorica}, 37(4):555--572, 2017.

\bibitem{baranygeneralization1982}
I.~B\'{a}r\'{a}ny.
\newblock A generalization of {C}arath\'{e}odory's theorem.
\newblock {\em Discrete Math.}, 40(2-3):141--152, 1982.

\bibitem{floriancolorful2019}
F.~Frick and S.~Zerbib.
\newblock Colorful coverings of polytopes and piercing numbers of colorful
  {$d$}-intervals.
\newblock {\em Combinatorica}, 39(3):627--637, 2019.

\bibitem{furedi}
Z.~F{\"u}redi.
\newblock Maximum degree and fractional matchings in uniform hypergraphs.
\newblock {\em Combinatorica}, 1(2):155--162, 1981.

\bibitem{galeequilibrium1984}
D.~Gale.
\newblock Equilibrium in a discrete exchange economy with money.
\newblock {\em Internat. J. Game Theory}, 13(1):61--64, 1984.

\bibitem{HolmsenIntersection2016}
A.~F. Holmsen.
\newblock The intersection of a matroid and an oriented matroid.
\newblock {\em Adv. Math.}, 290:1--14, 2016.

\bibitem{Kaiser}
T.~Kaiser.
\newblock Transversals of {$d$}-intervals.
\newblock {\em Discrete Comput. Geom.}, 18(2):195--203, 1997.

\bibitem{knaster1929}
B.~Knaster, C.~Kuratowski, and S.~Mazurkiewicz.
\newblock Ein beweis des fixpunktsatzes f\"ur n-dimensionale simplexe.
\newblock {\em Fund. Math.}, 14(1):132--137, 1929.

\bibitem{Komiyasimple1994}
H.~Komiya.
\newblock A simple proof of {K}-{K}-{M}-{S} theorem.
\newblock {\em Econom. Theory}, 4(3):463--466, 1994.

\bibitem{matouseklower2001}
J.~Matou\v{s}ek.
\newblock Lower bounds on the transversal numbers of {$d$}-intervals.
\newblock {\em Discrete Comput. Geom.}, 26(3):283--287, 2001.

\bibitem{NSZ}
K.~Nyman, F.~E. Su, and S.~Zerbib.
\newblock Fair division with multiple pieces.
\newblock {\em Discrete Appl. Math.}, 283:115--122, 2020.

\bibitem{shapleybalanced1973}
L.~S. Shapley.
\newblock On balanced games without side payments.
\newblock In {\em Mathematical programming ({P}roc. {A}dvanced {S}em., {U}niv.
  {W}isconsin, {M}adison, {W}is., 1972)}, pages 261--290. Math. Res. Center
  Publ., No. 30, 1973.

\bibitem{shihcombinatorial1993}
M.~Shih and S.~Lee.
\newblock Combinatorial formulae for multiple set-valued labellings.
\newblock {\em Math. Ann.}, 296(1):35--61, 1993.

\bibitem{soberonRobust2018}
P.~Sober\'{o}n.
\newblock Robust {T}verberg and colourful {C}arath\'{e}odory results via random
  choice.
\newblock {\em Combin. Probab. Comput.}, 27(3):427--440, 2018.

\bibitem{soberon2021fair}
P.~Soberón.
\newblock Fair distributions for more participants than allocations.
\newblock {\em arXiv preprint arXiv:2110.03600}, 2021.

\bibitem{spernerslemma1928}
E.~Sperner.
\newblock Neuer beweis f\"{u}r die invarianz der dimensionszahl und des
  gebietes.
\newblock {\em Abh. Math. Sem. Univ. Hamburg}, 6(1):265--272, 1928.

\bibitem{stromquisthow1980}
W.~Stromquist.
\newblock How to cut a cake fairly.
\newblock {\em Amer. Math. Monthly}, 87(8):640--644, 1980.

\bibitem{tardos1995}
G.~Tardos.
\newblock Transversals of {$2$}-intervals, a topological approach.
\newblock {\em Combinatorica}, 15(1):123--134, 1995.

\bibitem{woodalldividing1980}
D.~R. Woodall.
\newblock Dividing a cake fairly.
\newblock {\em J. Math. Anal. Appl.}, 78(1):233--247, 1980.

\end{thebibliography}

\end{document}